\documentclass[10pt]{article}
\usepackage{amssymb}
\usepackage{amsfonts}
\usepackage{amsmath,amssymb,amsthm,amsfonts,enumerate,graphicx,geometry,color}
\usepackage{subfigure}
\newtheorem{thm}{Theorem}[section]
\newtheorem{cor}[thm]{Corollary}
\newtheorem{lem}[thm]{Lemma}

\theoremstyle{definition}

\theoremstyle{remark}
\newtheorem{rem}{Remark}[section]

\numberwithin{equation}{section}

\newtheoremstyle{example}{\topsep}{\topsep}%
     {}%         Body font
     {}%         Indent amount (empty = no indent, \parindent = para indent)
     {\bfseries}% Thm head font
     {}%        Punctuation after thm head
     {\newline}%     Space after thm head (\newline = linebreak)
     {\thmname{#1}\thmnumber{ #2}\thmnote{ #3}}%         Thm head spec

\newtheorem{ex}{\bfseries \emph{Example}}[section]

\begin{document}

\title{Decomposition of Integral Self-Affine Multi-Tiles
\footnotetext{School of Mathematics and Statistics, Central China Normal University, P. R. China.}
\footnotetext{xiaoyefu@mail.ccnu.edu.cn}
 \footnotetext{Department of Mathematics and Statistics,
McMaster University, Hamilton, Ontario, L8S 4K1, Canada.}
\footnotetext{gabardo@mcmaster.ca}
\footnotetext{The research is supported in part by the NSFC 11401205}}
%\footnotetext{905-522-0935 (fax)
%xiaoyefu@gmail.com; gabardo@mcmaster.ca}}

\author{Xiaoye Fu and Jean-Pierre Gabardo}

\date{}
\maketitle

%\begin{abstract}

{\bf Abstract.} In this paper, we propose a method to decompose  an integral self-affine ${\mathbb Z}^n$-tiling set
$K$ into measure disjoint pieces $K_j$ satisfying
$K=\displaystyle\bigcup K_j$ in such a way that the collection of sets $K_j$ forms an integral self-affine
collection associated with the matrix $B$ and this with a minimum number of pieces $K_j$.
When used on a given measurable $\mathbb{Z}^n$-tiling
set $K\subset\mathbb{R}^n$, this decomposition terminates after finitely many
steps if and only if the set $K$ is an integral self-affine multi-tile. Furthermore,
we show that the minimal decomposition we provide is unique.

{\bf Key words.} self-affine collection, tiling sets, self-affine multi-tiles

{\bf AMS subject classifications.} 52C22, 28A80
%\end{abstract}

\section{Introduction}

{\hskip 0.7cm}  In the relevant literature, it has been shown that
integral self-affine (multi-)tiles
 can be used to construct Haar-type wavelet and multiwavelet basis.
 They also arise in the construction of certain compactly supported wavelets and multiwavelets
in the frequency domain, with the representations of their set-valued equations
playing an essential role in this procedure (see \cite{BS02,GH, FG,GYU}).

 The set-valued equation of an integral self-affine tile is
uniquely determined. However, for an integral self-affine multi-tile,
the set-valued equation defining it may not be unique as shown in \cite{GHA}.
Suppose that $B\in M_n(\mathbb{Z})$, the set $n\times n$ matrices with integer entries,
and that $B$ is expansive. i.e.~all of its eigenvalues have moduli greater than one.
For a given measurable set $K\subset \mathbb{R}^n$,
we can determine whether or not $K$ is an integral self-affine tile by
checking if it satisfies a set-valued
equation. However, this is not the case for integral self-affine
multi-tiles since we would need to know a priori  a measure disjoint partition of $K$
to make this determination. This yields a natural question:

\

 Given a measurable set $K\subset \mathbb{R}^n$ and an associated expansive matrix
$B\in M_n(\mathbb{Z})$, how can we verify whether or not $K$ is an integral
self-affine multi-tile associated with $B$? In the case of a positive answer,
the set $K$ could then be  used  for constructing
associated Haar-type wavelets.

\

We first recall the definition of  integral self-affine multi-tiles
associated with a given  expansive matrix with integral entries.

\bigskip

A finite collection of sets $ K_i\subseteq
\mathbb{R}^n, 1\le i \le M$, that are either
compact or empty, is said to be {\it an integral self-affine collection} if there is an expansive
 matrix $B\in M_n(\mathbb{Z})$ and finite (possibly
empty) sets $\Gamma_{ij}\subseteq\mathbb{Z}^n, \ i,j=1,\dotso,M,$ such that
\begin{equation}\label{e-1-1}
BK_i=\displaystyle\bigcup_{j=1}^{M}(\Gamma_{ij}+K_j)\  \text{for} \
i=1,\dotsc,M,
\end{equation}
and for any $i,j,k\in\{1,\dotsc,M\},$
\begin{equation}\label{e-1-2}
(\beta+K_i)\cap(\gamma+K_j)\cong\emptyset\ \text{for} \
\beta\in\Gamma_{ki}, \ \gamma\in\Gamma_{kj} \ \text{and} \  i\ne j \ \text{or} \
\beta\ne\gamma,
\end{equation}
where for two measurable sets $E, F\subseteq\mathbb{R}^n$, $E \cong F$ means that the symmetric difference $(F\setminus E)\bigcup (E\setminus F)$ has zero Lebesgue measure. The set
$\Gamma:=\lbrace\Gamma_{ij}\rbrace_{1\le i,j\le M}$ is called {\it a collection of digit sets} and it is called {\it a standard collection of digit
sets} if for each
$j\in\{1,\dotsc,M\},$ $\mathcal{D}_j:=\displaystyle\bigcup_{i=1}^{M}\Gamma_{ij}$ is a
complete set of coset representatives for the group
$\mathbb{Z}^n/B\mathbb{Z}^n$ \cite{GHA,GM,LWN}.

\bigskip

A finite collection of sets $ \{K_i\subseteq
\mathbb{R}^n, 1\le i \le M\}$ is said to {\it $\Lambda$-tile $\mathbb{R}^n$} or to be {\it a $\Lambda$-tiling set} (see \cite{GHA}), if
$
K_i \cap K_j \cong \emptyset \ \text{for} \ i \ne j \in \{1, \dotsc, M\},
$
and for $K:= \bigcup\limits_{i=1}^M K_i$, we have
\begin{itemize}
 \item  $\bigcup_{\ell\in\Lambda}(\ell+K) \cong \mathbb{R}^n$,
 \item  $(K+\ell_1)\cap(K+\ell_2) \cong \emptyset \ \rm{for} \  \ell_1,\ell_2 \in\Lambda \ \rm{and} \ \ell_1\ne\ell_2$.
\end{itemize}

If an integral self-affine collection $\lbrace K_i\subset\mathbb{R}^n, 1\le i\le M\rbrace$ $\Lambda$-tiles $\mathbb{R}^n$,
 we call $K : = \bigcup\limits_{i=1}^M K_i$ {\it an integral self-affine $\Lambda$-tiling set
 with $M$ prototiles}, or {\it an integral self-affine multi-tile} for short.
 $K$ is called an integral self-affine tile if $M=1$.
It is known that (\ref{e-1-1}) has exactly one solution \cite{H} if $M=1$,
i.e.~$K\subset\mathbb{R}^n$ is uniquely determined by the
associated $n\times n$ expansive matrix $B\in M_n(\mathbb{Z})$ and
the digit set $\Gamma$ and, conversely, if $B$ and $K$ are fixed,
the set of digits satisfying (\ref{e-1-1}) is unique.
The geometry, topological structure and
tiling property of
 self-affine tiles \cite{F, K,LR,LLR1,LLR13,LRT,HL,LWG,LWN,LWA} and the connections to wavelet theory  have been studied quite extensively (\cite{GHS, FG, GYU, GM, LW}).

\bigskip

In contrast, if $M>1$ and   $B$ and $K$ are fixed, there may exist several sets of digits
for which Equation (\ref{e-1-1}) holds.
The  study of self-affine multi-tiles is not as well developed
because of their complicated structure.
 For a detailed study of the general solutions of
(\ref{e-1-1}) and  the relationship between multi-wavelets
and the theory of integral self-affine multi-tiles, we  refer the reader to
\cite{GHA, FW, LW03}. As another application to wavelet theory,
we considered the problem of constructing wavelet sets using
integral self-affine multi-tiles in the frequency domain \cite{FG}.
The wavelet sets constructed in \cite{FG} depends in an essential way on the structure of
 the associated integral self-affine multi-tiles. The decomposition (or the representation)
of self-affine multi-tiles allowed us to obtain rather explicit descriptions
of  the wavelet sets we constructed there.  This motivated us to consider
the problem of decomposing a measurable set into a  self-affine collection
in the best possible way if it is an integral self-affine multi-tile.

\bigskip

It is known that an integral self-affine multi-tile must be
a $\mathbb{Z}^n$-tiling set if it can be used to construct a Haar-type multiwavelet \cite{GHA, FW}.
Thus, we will restrict our attention to $\mathbb{Z}^n$-tiling sets and
we will consider the problem of representing a $\mathbb{Z}^n$-tiling set
which is an integral self-affine multi-tile as the union of an integral
self-affine collection with the minimal number of prototiles,
since, as we mentioned before, the representation of an integral self-affine multi-tile is not
necessarily unique.
For example, in dimension one, if $K=[-\frac{1}{4},\frac{3}{4}]$ associated with $B=2$,
then $K$ is not only an integral self-affine $\mathbb{Z}$-tiling set with 3 prototiles,
 but also an integral self-affine $\mathbb{Z}$-tiling set with 4 prototiles (see section 3). The main goal
of this paper is to provide a method to decompose an integral self-affine multi-tile $K$ which is a
$\mathbb{Z}^n$-tiling set associated with some expansive matrix $B\in M_n(\mathbb{Z})$
into distinct measure disjoint
pieces $K_j$  such that the collection of sets
$K_j$ is an integral self-affine collection associated with
$B$ and the number of prototiles is minimal. Moreover, this
minimal representation will be shown to be unique. We will show that the proposed
decomposition will terminate in finitely many steps only when the measurable $\mathbb{Z}^n$-tiling set
to which it is applied
 is in fact  an integral self-affine multi-tile associated with the
expansive matrix $B\in M_n(\mathbb{Z})$.

\bigskip

The paper is organized as follows. In Section 2, we provide a method to represent a
$\mathbb{Z}^n$-tiling set which is an integral self-affine multi-tile as a
 union of prototiles with the least number of prototiles.
In section 3, we give some examples to
illustrate our method given in Section 2. Moreover, we construct some examples
showing that some wavelet sets cannot be constructed by the method
in \cite{FG} using integral self-affine multi-tiles.

\section{The representation of integral self-affine multi-tiles}

{\hskip 0.7cm}
In the following, $\lvert K\rvert$ denotes the Lebesgue measure of a measurable set $K\subset\mathbb{R}^n$. We will assume that $K:=\displaystyle\bigcup_{i=1}^{M}K_i$ with $\lvert K_i\rvert>0$ is an integral self-affine ${\mathbb Z}^n$-tiling set with $M$ prototiles satisfying (\ref{e-1-1}) and (\ref{e-1-2}). Moreover, we assume here that for each $j\in S$, $\mathcal{D}_j:=\displaystyle\bigcup_{i=1}^{M}\Gamma_{ij}$ is a complete set of coset representatives for $\mathbb{Z}^n/B\mathbb{Z}^n$. This condition is known to be necessary in order for an integral self-affine multi-tile $K$ to be a $\mathbb{Z}^n$-tiling set \cite{GHA}.

\medskip

Define $\Gamma_{ij}^m\subset\mathbb{Z}^n$, for $m\ge 1$, by
\begin{eqnarray}\label{e-2-1}
B^mK_i=\displaystyle\bigcup_{j=1}^{M}(K_j+\Gamma_{ij}^m), \ i=1,\dotsc,M.
\end{eqnarray}
 It follows from (\ref{e-1-1}) that
$\Gamma_{ij}^1=\Gamma_{ij}$. Using (\ref{e-1-1}) with each prototile $K_j$, we have
$$
B^2K_i
=\displaystyle\bigcup_{\ell=1}^{M} (BK_{\ell}+B\Gamma_{i\ell})
= \displaystyle\bigcup_{\ell=1}^{M}\displaystyle\bigcup_{j=1}^{M} (K_j+\Gamma_{\ell j}+B\Gamma_{i\ell})
= \displaystyle\bigcup_{j=1}^{M} (K_j+\displaystyle\bigcup_{\ell=1}^{M} (\Gamma_{\ell j}+B\Gamma_{i\ell})).
$$
Thus,
$\Gamma_{ij}^2=\displaystyle\bigcup_{\ell=1}^{M}(\Gamma_{\ell j}+B\Gamma_{i\ell}).$
 Inductively, we obtain
\begin{eqnarray}\label{e-2-3}
\Gamma_{ij}^m=\displaystyle\bigcup_{\ell=1}^{M}(\Gamma_{\ell j}+B\Gamma_{i\ell}^{m-1}).
\end{eqnarray}
Define
\begin{eqnarray}\label{e-2-4}
\mathcal{D}_j^m:=\displaystyle\bigcup_{i=1}^{M}\Gamma_{ij}^m, \ m\ge 1.
\end{eqnarray}
Then the corresponding self-affine multi-tile $K=\displaystyle\bigcup_{i=1}^MK_i$ satisfies
\begin{eqnarray}\label{e-2-5}
BK=\displaystyle\bigcup_{i=1}^M BK_i
=\displaystyle\bigcup_{i=1}^M\displaystyle\bigcup_{j=1}^M (K_j+\Gamma_{ij})
=\displaystyle\bigcup_{j=1}^M (K_j+\mathcal{D}_j),
\end{eqnarray}
and, more generally,
\begin{eqnarray}\label{e-2-6}
B^mK=\displaystyle\bigcup_{i=1}^M B^m K_i
=\displaystyle\bigcup_{i=1}^M\displaystyle\bigcup_{j=1}^M (K_j+\Gamma_{ij}^m)
=\displaystyle\bigcup_{j=1}^M (K_j+\mathcal{D}_j^m), \ m\ge 1.
\end{eqnarray}

Denote $S: = \{1, \dotsc, M\}$. Define, for each $m\ge 1$, an equivalence relation $\overset{m}{\sim}$ on  $S$ by
\begin{eqnarray*}
i\overset{m}{\sim} j\Leftrightarrow \mathcal{D}_i^k=\mathcal{D}_j^k, \ 1\le k\le m, \ \text{for} \ i, j \in S.
\end{eqnarray*}
For each $m\ge 1$, we will denote by $F_1^m,\dotsc, F_{S_m}^m$ the corresponding equivalence classes. Obviously, $F_1^m,\dotsc, F_{S_m}^m$ give a partition of $S$.
If $\mathcal{D}_i^m=\mathcal{D}_j^m$ for any $m\ge 1$, we say that $i$ is equivalent to $j$ and denote as $i\sim j$. According to this equivalence relation, we can get a
partition $\{F_j\}_{j=1}^{\ell},\ 1<\ell\le M$, of $S$, where the sets $F_j, \ 1\le j\le \ell$ are the corresponding equivalence classes. Then, we have the following result.

\medskip

\begin{lem}\label{l-2-2}
Let $K=\bigcup\limits_{i=1}^{M}K_i$ be an integral self-affine $\mathbb{Z}^n$-tiling set with $M$ prototiles and
let $\widetilde{K}_j=\bigcup\limits_{i\in F_j}K_i, \ j=1,\dotsc, \ell$, where the partition  $\{ F_j\}_{j=1}^{\ell}$ of $S$ is defined as above.
Then $\widetilde{K}: = \bigcup\limits_{j=1}^{\ell} \widetilde{K}_j$ is an integral self-affine $\mathbb{Z}^n$-tiling set with $\ell$ ($\le M$) prototiles.
\end{lem}

\medskip

\begin{proof}
Obviously, $\widetilde{K}: = \bigcup\limits_{j=1}^{\ell} \widetilde{K}_j = \bigcup\limits_{i=1}^{M}K_i$ is a ${\mathbb Z}^n$-tiling set by the assumption.
It is left to prove that the collection $\{\widetilde{K}_j, 1 \le j \le \ell\}$ is an integral self-affine collection.

\noindent First, we will prove that $\bigcup\limits_{i\in F_s}\Gamma_{ij}$ does not depend on $j$ for $j\in F_t$, $t\in\{1,\dotsc, \ell\}$.
Without loss of generality, we can assume that $i_1\ne i_2 \in F_t$ for some $t\in\{1,\dotsc, \ell\}$. Then for any $m\ge 1$,
using (\ref{e-2-3}), (\ref{e-2-4}) and that $S=\bigcup\limits_{j=1}^{\ell}F_j$, we have
\begin{eqnarray}\label{e-2-12}
\begin{array}{crl}
\mathcal{D}_{i_1}^{m+1}=\mathcal{D}_{i_2}^{m+1}
&\Longleftrightarrow&
\bigcup\limits_{i=1}^{M} (\Gamma_{ii_1}+B\mathcal{D}_i^{m}) = \bigcup\limits_{i=1}^{M} (\Gamma_{ii_2}+B\mathcal{D}_i^{m})\\
&\Longleftrightarrow&
\bigcup\limits_{s=1}^{\ell}\bigcup\limits_{i\in F_s} (\Gamma_{ii_1}+B\mathcal{D}_i^{m})
=\bigcup\limits_{s=1}^{\ell}\bigcup\limits_{i\in F_s} (\Gamma_{ii_2}+B\mathcal{D}_i^{m}).
\end{array}
\end{eqnarray}
Let $\phi_{si_1}=\bigcup\limits_{i\in F_s}\Gamma_{ii_1},\ \phi_{si_2}=\bigcup\limits_{i\in F_s}\Gamma_{ii_2}$. Next, we want to show that $\phi_{si_1}=\phi_{si_2}$ for any $i_1, \ i_2\in F_t$.  Note that
\begin{eqnarray}\label{e-2-13}
\mathcal{D}_{i_1}=\bigcup\limits_{s=1}^{\ell}\phi_{si_1}=\bigcup\limits_{s=1}^{\ell}\phi_{si_2}=\mathcal{D}_{i_2},
\end{eqnarray}
 which is a complete set of coset representatives for $\mathbb{Z}^n/B\mathbb{Z}^n$ by assumption.
By the definition of $F_s$, the sets $\mathcal{D}_i^m$ are the same for $i\in F_s$ and any $m\ge 1$. We will denote this common set by
 $\mathcal{D}_{\tau_s}^m$. Then equation (\ref{e-2-12}) can be rewritten as the following
\begin{eqnarray}\label{e-2-14}
\mathcal{D}_{i_1}^{m+1}=\mathcal{D}_{i_2}^{m+1}\Longleftrightarrow \bigcup\limits_{s=1}^{\ell} (\phi_{si_1}+B\mathcal{D}_{\tau_s}^{m})
=\bigcup\limits_{s=1}^{\ell} (\phi_{si_2}+B\mathcal{D}_{\tau_s}^{m}).
\end{eqnarray}
Suppose that, for some $s\in \{1,\dotsc,\ell\}$, $\phi_{si_1}\ne\phi_{si_2}$. Then $\phi_{si_1}\setminus\phi_{si_2}\ne\emptyset$ or
$\phi_{si_2}\setminus\phi_{si_1}\ne\emptyset$.

\noindent WLOG, we assume that there exists $x\in\phi_{si_1}\setminus\phi_{si_2}$, then
$x\in \phi_{ti_2}$ for some $t\in\{1,2,\dotsc,\ell\}$ and $t\ne s$ by (\ref{e-2-13}).
Next, we will prove that
$(x+B\mathcal{D}_{\tau_s}^{m}) \cap (y+B\mathcal{D}_{\tau_w}^{m}) = \emptyset$ for any $y\in \phi_{wi_2}$, $w\in \{1,2,\dotsc, \ell\}$ and $y\ne x$.
Otherwise, there is $y\in\phi_{wi_2}$ and $y\ne x$ such that
$(x+B\mathcal{D}_{\tau_s}^{m}) \cap (y+B\mathcal{D}_{\tau_w}^{m})\ne\emptyset$.
This implies that $(x-y)\in B\mathbb{Z}^n$, which gives a contradiction
since $x\in \phi_{si_1}\subset\mathcal{D}_{i_1}, \ y\in \phi_{wi_2}\subset\mathcal{D}_{i_2}$ and $\mathcal{D}_{i_1}=\mathcal{D}_{i_2}$
is a complete set of coset representatives for $\mathbb{Z}^n/B\mathbb{Z}^n$ by (\ref{e-2-13}).
Hence, it follows from (\ref{e-2-14}) that the only possibility is $x+B\mathcal{D}_{\tau_s}^{m} = x+B\mathcal{D}_{\tau_t}^{m}$,
which forces that $\mathcal{D}_{\tau_s}^{m}=\mathcal{D}_{\tau_t}^{m}$ for any $m\in\mathbb{N}$, contradicting the fact that $F_s$ and $F_t$ are different equivalence classes.
Therefore, $\phi_{si_1}=\phi_{si_2}$ for any $i_1, \ i_2\in F_t$. In other words,  $\bigcup\limits_{i\in F_s}\Gamma_{ij}$ is independent of $j\in F_t$. Let
$\Lambda_{st}=\bigcup\limits_{i\in F_s}\Gamma_{ij},\ j\in F_t$.  Then
\begin{eqnarray*}
B\widetilde{K}_s=B\bigcup\limits_{i\in F_s}K_i
=\bigcup\limits_{i\in F_s}\bigcup\limits_{j=1}^M (K_j+\Gamma_{ij})
= \bigcup\limits_{t=1}^{\ell}\bigcup\limits_{j\in F_t} (K_j+\bigcup\limits_{i\in F_s}\Gamma_{ij})
= \bigcup\limits_{t=1}^{\ell} (\widetilde{K}_t+\Lambda_{st}).
\end{eqnarray*}
This proves that $\{\widetilde{K}_j, 1\le j \le \ell\}$ is an integral self-affine collection.
\end{proof}

Lemma \ref{l-2-2} provides us an idea to decompose $K$ according to the equivalence relation ``$\sim$".
 We can decompose $K$ by finding the intersection of non-zero measure between different integer translations of its $B$-dilations and itself.

\medskip

For $m\ge 1$, let
$C_m$ be the collection of  sets with positive Lebesgue measure of the form
 $$(B^{k_1}K+\ell_1)\cap(B^{k_2}K+\ell_2)\cap \dotsm \cap (B^{k_r}K+\ell_r)\cap K,$$
where $1\le k_i\le m$, $\ell_i\in \mathbb{Z}^n$, $r\ge 1$. Then $C_1\subseteq C_2\subseteq\dotsm$ and each $C_m$ is finite
since, for a fixed $k\ge1$, only finitely integral translates of $B^kK$ can intersect $K$.
Furthermore, $C_m$ is stable under intersection, i.e. $A,B\in C_m$ implies that $A\cap B\in C_m$ if $A\cap B\ncong\emptyset$.

\medskip

For each $m\ge 1$, we order the sets in $C_m$ by inclusion
and we denote by $C_m^{\prime}$ the subset of $C_m$
consisting of all \emph{minimal} sets in  $C_m$.
``Minimal sets in $C_m$" means that $A\in C_m^{\prime}$ if and only if $A\in C_m$ and $A$ has the property that, if $B\in   C_m$
and $B\subset A$, then $B=A$.
Then the elements of $C_m^{\prime}$ are of the form
 \begin{eqnarray}\label{e-2-e}
\bigcap\limits_{\ell_m\in L_m}\dotsm\bigcap\limits_{\ell_1\in L_1}(B^mK+\ell_m)\cap\dotsc\cap(BK+\ell_1)\cap K
\end{eqnarray}
for some subsets $L_i \subset \mathbb{Z}^n, \ i=1,\dotsc, m$,
 and any element of $C_m$ can be written as a union of sets in $C_m^{\prime}$.
Clearly, the collection of sets in  $C_m^{\prime}$ forms a partition of $K$. We should mention here that the partitions of $K$ considered in this work are always meant up to intersections of zero-measure.

\medskip

\begin{lem}\label{l-2-1}
Let $K$ be an integral self-affine $\mathbb{Z}^n$-tiling set with $M$ prototiles. For each $m\ge 1$, a set $E\in C_m^{\prime}$ if and only if $E=\bigcup\limits_{i\in F_s^m}K_i$, for some $s$ with $1\le s\le S_m$, where  $F_s^m$ is some equivalence class according to the equivalence relation  $\overset{m}\sim$ in the $m$-th step and $S_m$ is the number of equivalence class in this step.
\end{lem}
\begin{proof}

%By assumption, $K=\bigcup\limits_{j=1}^M K_j$ is an integral self-affine $\mathbb{Z}^n$-tiling set with $M$ prototiles which satisfies (\ref{e-1-1}) and (\ref{e-1-2}).

\medskip

To prove that $E=\bigcup\limits_{i\in F_s^m}K_i$, for some $1\le s\le S_m$ implies that $E\in C_m^{\prime}$.
We will first prove that if $i_1\in S$, $K_{i_1}\subset E\in C_m^{\prime}$ and $i_2\overset{m}\sim i_1,$ then $K_{i_2}\subset E$. Indeed, any $E\in C_m^{\prime}$ can be written as
\begin{eqnarray*}
E&=&\bigcap\limits_{\ell_m\in L_m}\dotsm\bigcap\limits_{\ell_1\in L_1}(B^mK+\ell_m)\cap\dotsc\cap(BK+\ell_1)\cap K\\
&=&\bigcap\limits_{\ell_m\in L_m}\dotsm\bigcap\limits_{\ell_1\in L_1}\bigcup\limits_{j_0,j_1,\dotsc, j_m=1}^M(K_{j_m}+\mathcal{D}_{j_m}^m+\ell_m)\cap\dotsc\cap(K_{j_1}+\mathcal{D}_{j_1}+\ell_1)\cap K_{j_0}\\
&=&\bigcap\limits_{\ell_m\in L_m}\dotsm\bigcap\limits_{\ell_1\in L_1}\bigcup\limits_{j=1}^M(K_j+\mathcal{D}_j^m+\ell_m)\cap\dotsc\cap(K_j+\mathcal{D}_j+\ell_1)\cap K_j.
\end{eqnarray*}
It follows that the inclusion $K_{i_1}\subset E$ is equivalent to
$L_1\subseteq -\mathcal{D}_{i_1}^1, L_2\subseteq -\mathcal{D}_{i_1}^2,\dotsc,  L_m\subseteq -\mathcal{D}_{i_1}^m$. Thus, if $i_2\overset{m}\sim i_1,$ we have $K_{i_2}\subset E\Leftrightarrow K_{i_1}\subset E$.

\noindent  Next, we will prove that $E\in C_m^{\prime}$ implies that  $E=\bigcup\limits_{i\in F_s^m}K_i$, for some $1\le s\le S_m$. Using the definition of $C_m$ and the fact that $K$ is an integral self-affine $\mathbb{Z}^n$-tiling set, it is easy to see that any element of $C_m$, and thus of $C_m^{\prime}$, is the union of some sets $K_i, \ i\in I\subseteq S$. In the following , we will show that $I = F_s^m$ for some $1\le s\le S_m$.
we use induction on $m$.
For $m=1$, suppose that $\bigcap\limits_{\ell_1\in L_1}(BK+\ell_1)\cap K\in C_1^{\prime}$. Then there exists some $T_1\subseteq S$ such that
\begin{eqnarray*}
\bigcap\limits_{\ell_1\in L_1}(BK+\ell_1)\cap K=\bigcap\limits_{\ell_1\in L_1}\bigcup\limits_{j=1}^M(K_j+\mathcal{D}_j+\ell_1)\cap K_j=\bigcup\limits_{i\in T_1}K_i,
\end{eqnarray*}
which implies that $L_1\subseteq \bigcap\limits_{i\in T_1}-\mathcal{D}_i$. If $T_1$ contains at least two different elements, then $\mathcal{D}_{i_1}=\mathcal{D}_{i_2}$ for any $i_1,i_2\in T_1$. Otherwise, we can find $p\in-(\mathcal{D}_{i_1}\setminus \mathcal{D}_{i_2})$ and we have
\begin{eqnarray*}
K_{i_1}\subseteq (BK+p)\cap K\in C_1 \ {\rm{and}} \ K_{i_2}\cap (BK+p)\cap K \cong \emptyset.
\end{eqnarray*}
Thus, we obtain that
\begin{eqnarray*}
\emptyset\ncong K_{i_1}\subseteq \bigcap\limits_{\ell_1\in L_1}(BK+\ell_1)\cap K\cap (BK+p)=\bigcup\limits_{i\in T_1}K_i \cap (BK+p)\cap K\subseteq \bigcup\limits_{i\in T_1\setminus\{i_2\}}K_i,
\end{eqnarray*}
which contradicts the fact that $\bigcap\limits_{\ell_1\in L_1}(BK+\ell_1)\cap K=\bigcup\limits_{i\in T_1}K_i$ is the minimal set. This proves our claim for $m=1$, i.e., that $I = F_{s}^1 $ for some $1\le  s \le S_1$. Suppose that our claim is true for some $m>1$. Then we will prove it is true for $m+1$. Let
$$\bigcap\limits_{\ell_{m+1}\in L_{m+1}}\dotsm\bigcap\limits_{\ell_1\in L_1}(B^{m+1}K+\ell_{m+1})\cap\dotsc\cap(BK+\ell_1)\cap K=\bigcup\limits_{i\in T_{m+1}\subseteq S}K_i\in C_{m+1}^{\prime}.$$
Since any element of $C_{m+1}^{\prime}$ is contained in some element of $C_m^{\prime}$, we here have $\bigcup\limits_{i\in T_{m+1}}K_i\subseteq \bigcup\limits_{i\in T_m} K_i\in C_m^{\prime}$ for some set $T_m\subset S$. Next we will prove that $T_{m+1}$ is an equivalence class associated with $\overset{m+1}\sim$. Using our induction hypothesis, $T_{m}$ is an equivalence class associated with $\overset{m}\sim$. If $T_{m+1}$ contains only one element, we are done. Assume that $T_{m+1}$ contains at least two different elements, say $i_1\ne i_2$, and $\mathcal{D}_{i_1}^{m+1}\ne \mathcal{D}_{i_2}^{m+1}$. Then there exists $p\in-(\mathcal{D}_{i_1}^{m+1}\setminus \mathcal{D}_{i_2}^{m+1})$ such that
\begin{eqnarray*}
K_{i_1}\subseteq (B^{m+1}K+p)\cap K\in C_{m+1} \ {\rm{and}} \ K_{i_2}\cap (B^{m+1}K+p)\cap K\cong\emptyset.
\end{eqnarray*}
Hence, we have
\begin{eqnarray*}
K_{i_1}&\subseteq&\bigcap\limits_{\ell_{m+1}\in L_{m+1}}\dotsm\bigcap\limits_{\ell_1\in L_1}(B^{m+1}K+\ell_{m+1})
\cap\dotsc\cap(BK+\ell_1)\cap K\cap (B^{m+1}K+p)\\
&\subseteq& \bigcup\limits_{i\in T_{m+1}\setminus\{i_2\}}K_i,
\end{eqnarray*}
which implies that $\bigcup\limits_{i\in T_{m+1}}K_i$ is not a minimal one in $C_m$. This is a contradiction. So $\mathcal{D}_{i_1}^{m+1}=\mathcal{D}_{i_2}^{m+1}$. On the other hand, since $i_1\overset{m}\sim i_2$ for any $i_1,i_2\in T_{m+1}\subseteq T_m$, we have $\mathcal{D}_{i_1}^{k}=\mathcal{D}_{i_2}^{k}$ for $k=1,\dotsc, m$. This proves that $T_{m+1}$ is an equivalence class for the equivalence under $\overset{m+1}\sim$.
\end{proof}

As we mentioned in the introduction, the representation of an integral self-affine $\mathbb{Z}^n$-tiling set is not unique. For convenience, we call the representation of an integral self-affine tiling set $K=\bigcup\limits_{i=1}^{S_M} K_i$ to be {\it in its simplest form} if the number $S_M$ is the minimal.

\medskip

Given an integral self-affine $\mathbb{Z}^n$-tiling set $K$, the sets in $C_m^{\prime}$ form a partition of $K$. Let $C_m^{\prime}=\{K_{mi},i=1,\dotsc, S_m\}$. Since each set in $C_m$ is a finite union of the prototiles making up the set $K$, the cardinality of $C_m$ is bounded independently of $m$ and thus, for some $m_0\ge 1$, $C_m=C_{m_0}$ for $m\ge m_0$, and thus $C_m^{\prime}=C_{m_0}^{\prime}$ for $m\ge m_0$.
We let $S_{m_0}=N$ and $K_{m_0i}=W_i, \ i=1,\dotsc, N$. The procedure just described provides a method to decompose any integral self-affine $\mathbb{Z}^n$-tiling set into measure disjoint prototiles with a representation in its simplest form, as shown in the following theorem.

\begin{thm}\label{t-2-1}
Suppose that $K$ is an integral self-affine $\mathbb{Z}^n$-tiling set. Then $K=\bigcup\limits_{i=1}^N W_i$ is an integral self-affine $\mathbb{Z}^n$-tiling set with $N$ prototiles, where $W_i, \ i=1,\dotsc, N$ are defined as above. Furthermore, the representation $K=\bigcup\limits_{i=1}^NW_i$ is in its simplest form and the simplest form representation of $K$ is unique.
\end{thm}
\begin{proof}

\medskip

Since $K$ is an integral self-affine $\mathbb{Z}^n$-tiling set, there exists a partition $\{K_i\}_{i=1}^M$ of $K$ such that
$K=\bigcup\limits_{i=1}^M K_i$ is an integral self-affine $\mathbb{Z}^n$-tiling set with $M$ prototiles. By the definition of $W_i$, the sets $W_i, \ i=1,\dotsc N$ are measure disjoint and Lemma \ref{l-2-1} implies that $W_i=\bigcup\limits_{j\in T_i\subseteq S}K_j$ for each $i=1,\dotsc, N$, where the sets $T_i, i=1,\dotsc, N$ are the equivalence classes obtained by the equivalence relation ``$\sim$". Thus $S=\bigcup\limits_{i=1}^N T_i$. It follows from Lemma \ref{l-2-2} that the collection $\{W_i\}_{i=1}^N$ is an integral self-affine collection. Therefore, $K=\bigcup\limits_{i=1}^NW_i$ is an integral self-affine $\mathbb{Z}^n$-tiling set with $N$ prototiles. We note that the sets $W_i, \ i=1,\dotsc, N$ do not depend on the sets $K_i, \ i=1,\dotsc, M$. For any representation of the set $K$ which satisfies the integral self-affine conditions, the set $W_i$ is the union of some prototiles $K_i$ for each $i=1,\dotsc, N$. Thus, we have $N\le M$. This proves that the representation $K=\bigcup\limits_{i=1}^NW_i$ is
in its simplest form and the simplest form representation of $K$ is unique.
\end{proof}

Theorem \ref{t-2-1} shows that if an integral self-affine $\mathbb{Z}^n$-tiling set $K=\bigcup\limits_{i=1}^{M}K_i$ is not in its simplest form, then there exists a partition $\{F_j\}_{j=1}^{\ell}$ of the set $S$ with $F_{j_0}$ having at least two elements for at least one $j_0$, such that the collection $\{\bigcup\limits_{i\in F_j}K_i\}_{j=1}^{\ell}$ is an integral self-affine collection. Combining all above results, we get the following corollary, which provides a necessary and sufficient condition for an integral self-affine $\mathbb{Z}^n$-tiling set to be in its simplest form.

\begin{cor}\label{c-2-1}
Let $K=\bigcup\limits_{i=1}^{M}K_i$ be an integral self-affine $\mathbb{Z}^n$-tiling set with $M$ prototiles. Then, the representation $K=\bigcup\limits_{i=1}^{M}K_i$ is in its simplest form if and only if for any $i_1,\ i_2\in S$ with $i_1\ne i_2$, there exists some $m\ge 1$ such that
$\mathcal{D}_{i_1}^m\ne\mathcal{D}_{i_2}^m$.
\end{cor}

\medskip

Theorem \ref{t-2-1} gives us an iterative method to decompose an integral self-affine $\mathbb{Z}^n$-tiling set $K$ into measure disjoint prototiles $K_j, \ j=1,\dotsc, M$ such that the collection $\{K_j\}_{j=1}^M$ is an integral self-affine collection. Moreover, this representation is in its simplest form and the decomposition is unique in the sense that the number of prototiles is least by Corollary \ref{c-2-1} and also, in the sense, that given any representation of $K$ as a union of prototiles $K_j, \ j=1,\dotsc, M$, the elements of the minimal representation can always be written as finite unions of these sets $K_j$. Given an integral self-affine $\mathbb{Z}^n$-tiling set $K$, we compute the collection $C_m^{\prime}$ step by step until we find an integer $m_0$ such that $C_m^{\prime}=C_{m_0}^{\prime}$ for any $m\ge m_0$. Or, alternatively, we check whether or not the collection $C_m^{\prime}$ obtained at each step is an integral self-affine collection. If it is, we stop and $K=\bigcup\limits_{i=1}^{S_m}
K_{mi}$ is an integral self-affine $\mathbb{Z}^n$-tiling set and the representation is in its simplest form.

\medskip

As we mentioned before, there are many integral self-affine $\mathbb{Z}^n$-tiling sets which have different representations. However, the representation we provide here is unique by Corollary \ref{c-2-1}. Such examples will be given in Section 3. Furthermore, we can also use this iterative method to determine whether or not a given measurable $\mathbb{Z}^n$-tiling set $K\subset\mathbb{R}^n$ is an integral self-affine multi-tile associated with any given $n\times n$ integral expansive matrix $B$. $K$ is an integral self-affine $\mathbb{Z}^n$-tiling set if and only if the process stops after finitely many steps.

\medskip

Finally, we mention that our iterative method
 can only be applied to  measurable $\mathbb{Z}^n$-tiling sets.
 For a non $\mathbb{Z}^n$-tiling set, it is not applicable.

\section{Examples}

{\hskip 0.7cm}  For some integral self-affine tiling sets with simple geometrical shape, it is easy to see how to decompose the given measurable set $K\subset\mathbb{R}^n$ into measure disjoint pieces $K_j$ such that $K_j,\ j\in I$, where $I$ is a finite set, is an integral self-affine collection. However, for those with complicated geometrical shape, it might not be easy to represent it as an integral self-affine collection. For such self-affine tiling sets, we can use the method introduced in Section 2 to solve this problem. In this section, we will give some examples to show how to use the method given in Section 2 to represent an integral self-affine $\mathbb{Z}^n$-tiling set in its simplest form.

\medskip

\begin{ex}
In dimension one, consider the set $K=[-\frac{3}{4},\frac{1}{4}]$ associated with $B=2$.
\end{ex}

\medskip

The set $K$ here can be not only an integral self-affine $\mathbb{Z}$-tiling set with $4$ prototiles, but an integral self-affine $\mathbb{Z}$-tiling set with $3$ prototiles. In the following, we will give its representation for each case. For the first case, let
\begin{eqnarray*}
K_1=[-\frac{3}{4},-\frac{1}{2}], \ K_2=[-\frac{1}{2},-\frac{1}{4}], \ K_3=[-\frac{1}{4},0], \ K_4=[0, \frac{1}{4}].
\end{eqnarray*}
%Define $\Gamma_{ij}, \ i,j=1,2,3,4$ as in (\ref{e-1-1}) and $\mathcal{D}_j:=\bigcup\limits_{i=1}^4\Gamma_{ij}$.
Then, we have
\begin{align*}
BK_1&=[-\frac{3}{2},-1]=(K_2-1)\cup (K_3-1)\Longrightarrow \Gamma_{11}=\emptyset, \ \Gamma_{12}=\{-1\}, \ \Gamma_{13}=\{-1\}, \ \Gamma_{14}=\emptyset,\\
 BK_2&=[-1,-\frac{1}{2}]=(K_4-1)\cup K_1 \Longrightarrow \Gamma_{21}=\{0\}, \ \Gamma_{22}=\emptyset, \ \Gamma_{23}=\emptyset, \ \Gamma_{24}=\{-1\},\\
BK_3&=[-\frac{1}{2},0]=K_2\cup K_3 \Longrightarrow \Gamma_{31}=\emptyset, \ \Gamma_{32}=\{0\}, \ \Gamma_{33}=\{0\}, \ \Gamma_{34}=\emptyset,\\
BK_4&=[0,\frac{1}{2}]=(K_1+1)\cup K_4 \Longrightarrow \Gamma_{41}=\{1\}, \ \Gamma_{42}=\emptyset, \ \Gamma_{43}=\emptyset, \ \Gamma_{44}=\{0\},
\end{align*}
and
\begin{eqnarray*}
&&\mathcal{D}_1=\bigcup\limits_{i=1}^4\Gamma_{i1}=\Big\{0,1\Big\},  \  \ \ \mathcal{D}_2=\bigcup\limits_{i=1}^4\Gamma_{i2}=\Big\{-1,0\Big\},\\
&& \mathcal{D}_3=\bigcup\limits_{i=1}^4\Gamma_{i3}=\Big\{-1,0\Big\}, \ \mathcal{D}_4=\bigcup\limits_{i=1}^4\Gamma_{i4}=\Big\{-1,0\Big\}.
\end{eqnarray*}
This shows that for each $j\in\{1,2,3,4\}$, $\mathcal{D}_j$ is a complete set of coset representatives for the group $\mathbb{Z}/2\mathbb{Z}$
and the set $K$ is an integral self-affine $\mathbb{Z}$-tiling set with $4$ prototiles. Define
$\mathcal{D}_j^m$ as in (\ref{e-2-4}). It follows from (\ref{e-2-3}) and (\ref{e-2-4}) that $\mathcal{D}_j^m=\bigcup\limits_{i=1}^4 (\Gamma_{ij}+2\mathcal{D}_i^{m-1})$. Note that for $i\in\{1,2,3,4\}$, $\Gamma_{i2}=\Gamma_{i3}$. Thus we get $\mathcal{D}_2^m=\mathcal{D}_3^m$ for any $m\ge 1$ by the definition of $\mathcal{D}_j^m$. On the other hand, since
\begin{eqnarray*}
\mathcal{D}_1^2=\bigcup\limits_{i=1}^4 (\Gamma_{i1}+2\mathcal{D}_i) = \Big\{-2,-1,0,1\Big\}=\mathcal{D}_2^2, \
\mathcal{D}_4^2=\bigcup\limits_{i=1}^4 (\Gamma_{i4}+2\mathcal{D}_i) = \Big\{-3,-2,-1,0\Big\},
\end{eqnarray*}
it follows that $\mathcal{D}_1\ne \mathcal{D}_2$, $\mathcal{D}_1^2\ne \mathcal{D}_4^2$ and $\mathcal{D}_2^2\ne \mathcal{D}_4^2$. The equivalence classes for the equivalence relation $\sim$ are thus $\{1\},\{2,3\}$ and $\{4\}$. By Theorem \ref{t-2-1}, $K=\bigcup\limits_{i=1}^4 K_i$ is not in ``the simplest form". By the proof in Theorem \ref{t-2-1}, we let
\begin{eqnarray*}
K_1^{\prime}=[-\frac{3}{4},-\frac{1}{2}], \ K_2^{\prime}=[-\frac{1}{2},0], \ K_3^{\prime}=[0,\frac{1}{4}].
\end{eqnarray*}
Define $\Gamma_{ij}^{\prime},\ i,j=1,2,3$ to satisfy $BK_i^{\prime}=\bigcup\limits_{j=1}^3 (K_j+\Gamma_{ij}^{\prime})$ and $\mathcal{D}_j^{\prime}=\bigcup\limits_{i=1}^3\Gamma_{ij}^{\prime}$. Then, we have
\begin{align*}
BK_1^{\prime}&=[-\frac{3}{2},-1]=(K_2^{\prime}-1)\Longrightarrow \Gamma_{11}^{\prime}=\emptyset, \ \Gamma_{12}^{\prime}=-1, \ \Gamma_{13}^{\prime}=\emptyset,\\
BK_2^{\prime}&=[-1,0]=K_1^{\prime}\cup K_2^{\prime}\cup (K_3^{\prime}-1)\Longrightarrow \Gamma_{21}^{\prime}=0, \ \Gamma_{22}^{\prime}=0, \ \Gamma_{23}^{\prime}=-1,\\
BK_3^{\prime}&=[0,\frac{1}{2}]=(K_1^{\prime}+1)\cup K_3^{\prime}\Longrightarrow \Gamma_{31}^{\prime}=1, \ \Gamma_{32}^{\prime}=\emptyset, \ \Gamma_{33}^{\prime}=0.
\end{align*}
Furthermore,
\begin{eqnarray*}
\mathcal{D}_1^{\prime}=\bigcup\limits_{i=1}^3\Gamma_{i1}^{\prime}=\Big\{0,1\Big\}, \ \mathcal{D}_2^{\prime}=\bigcup\limits_{i=1}^3\Gamma_{i2}^{\prime}=\Big\{-1,0\Big\}, \ \mathcal{D}_3^{\prime}=\bigcup\limits_{i=1}^3\Gamma_{i3}^{\prime}=\Big\{-1,0\Big\}.
\end{eqnarray*}
This shows that for each $i\in\{1,2,3\}$, $\mathcal{D}_i^{\prime}$ is a complete set of coset representatives for $\mathbb{Z}/2\mathbb{Z}$ and that $K=\bigcup\limits_{i=1}^3 K_i^{\prime}$ is an integral self-affine $\mathbb{Z}$-tiling set with $3$ prototiles. $\Box$

\medskip

\begin{rem}\label{r-3-1}
Let $K=\bigcup\limits_{j=1}^M K_j$ be an integral self-affine $\mathbb{Z}^n$-tiling set with $M$ prototiles. If $\Gamma_{ij_1}=\Gamma_{ij_2}$ for any $i\in S$, then $j_1\sim j_2$. But the converse is not necessarily true as shown in the next example.
\end{rem}

\bigskip

\begin{ex}\label{counter-3-1}
In dimension one, consider the set $K=[-\frac{3}{4},\frac{1}{4}]$ associated with $B=-3$.
\end{ex}

\medskip

Obviously, $K$ is a $\mathbb{Z}$-tiling set. For this example, the set $K$ can be represented as a union of many different kinds of prototiles.
The simplest form is the representation of $K$ as an integral self-affine tile, since
\begin{eqnarray*}
BK=[-\frac{3}{4},\frac{9}{4}]=K\cup(K+1)\cup(K+2).
\end{eqnarray*}
On the other hand, we can also let $K_1=[-\frac{3}{4},-\frac{1}{4}]$ and $K_2=[-\frac{1}{4},\frac{1}{4}]$, then
\begin{eqnarray*}
&& BK_1 = [\frac{3}{4},\frac{9}{4}] = (K_1+2)\cup (K_2+\{1,2\})\Longrightarrow \Gamma_{11} = \{2\}, \Gamma_{12} = \{1,2\}, \\
&&BK_2 = [-\frac{3}{4},\frac{3}{4}] = (K_1+\{0,1\}) \cup K_2\Longrightarrow \Gamma_{21} = \{0,1\}, \Gamma_{22} = \{0\},
\end{eqnarray*}
and $$\mathcal{D}_1=\bigcup\limits_{i=1}^2\Gamma_{i1}=\{0,1,2\}=\bigcup\limits_{i=1}^2\Gamma_{i2}=\mathcal{D}_2.$$
This shows that $K=\bigcup\limits_{i=1}^2K_i$ is an integral self-affine $\mathbb{Z}$-tiling set with $2$ prototiles.
Hence, $K$ is not only an integral self-affine $\mathbb{Z}$-tiling set with one prototiles,
but an integral self-affine $\mathbb{Z}$-tiling set with two prototiles.
Corollary \ref{c-2-1} shows that $\mathcal{D}_1^m=\mathcal{D}_2^m$ for any $m\ge 1$, i.e. $1\sim 2$.
However, $\Gamma_{i1}\ne \Gamma_{i2}$ for any $i=1,2$. $\Box$

\medskip

\begin{ex}\label{ex-3-2}
In dimension two, consider the set $K=H\cup(-H)\cup K^{\prime}$ associated with the matrix $B=\text{\scriptsize$\left(\begin{matrix}-1&1\\-3&1\end{matrix}\right)$}$, where
$$H=conv\Big\{\text{\scriptsize$\left(\begin{matrix}1/3\\1\end{matrix}\right)$},\text{\scriptsize$\left(\begin{matrix}2/3\\1\end{matrix}\right)$},\text{\scriptsize$\left(\begin{matrix}1/2\\1/2\end{matrix}\right)$},\text{\scriptsize$\left(\begin{matrix}1/6\\1/2\end{matrix}\right)$}\Big\} \ and \
K^{\prime}=conv\Big\{\text{\scriptsize$\left(\begin{matrix}-1/6\\1/2\end{matrix}\right)$},\text{\scriptsize$\left(\begin{matrix}1/2\\1/2\end{matrix}\right)$},\text{\scriptsize$\left(\begin{matrix}1/6\\-1/2\end{matrix}\right)$},\text{\scriptsize$\left(\begin{matrix}-1/2\\-1/2\end{matrix}\right)$}\Big\},$$
where $conv(E)$ denotes the convex hull of $E$.
\end{ex}

\medskip

It is easy to see that $K$ is a $\mathbb{Z}^2$-tiling set. The sets $K$ and $BK$ are depicted in Figure \ref{fig1.fig}. Clearly, we can divide $K$ into six pieces $\{K_j\}_{j=1}^6$ with $K_1=H, K_2=E, K_3=F, K_4=-E, K_5=-F, K_6=-H$, where
$$E=conv\Big\{\text{\scriptsize$\left(\begin{matrix}1/3\\0\end{matrix}\right)$},\text{\scriptsize$\left(\begin{matrix}0\\0\end{matrix}\right)$},\text{\scriptsize$\left(\begin{matrix}1/2\\1/2\end{matrix}\right)$},\text{\scriptsize$\left(\begin{matrix}1/6\\1/2\end{matrix}\right)$}\Big\} \ and \
F=conv\Big\{\text{\scriptsize$\left(\begin{matrix}-1/6\\ 1/2\end{matrix}\right)$},\text{\scriptsize$\left(\begin{matrix}1/6\\1/2\end{matrix}\right)$},\text{\scriptsize$\left(\begin{matrix}-1/3\\ 0\end{matrix}\right)$},\text{\scriptsize$\left(\begin{matrix}0\\0\end{matrix}\right)$}\Big\}.$$
Moreover, we have (see Figure \ref{fig1.fig})

\begin{figure}[hb!]
\centering
\hspace{-0.5in}
\includegraphics[trim=0mm 0mm 0mm 15mm, clip, scale=0.6,angle=0]{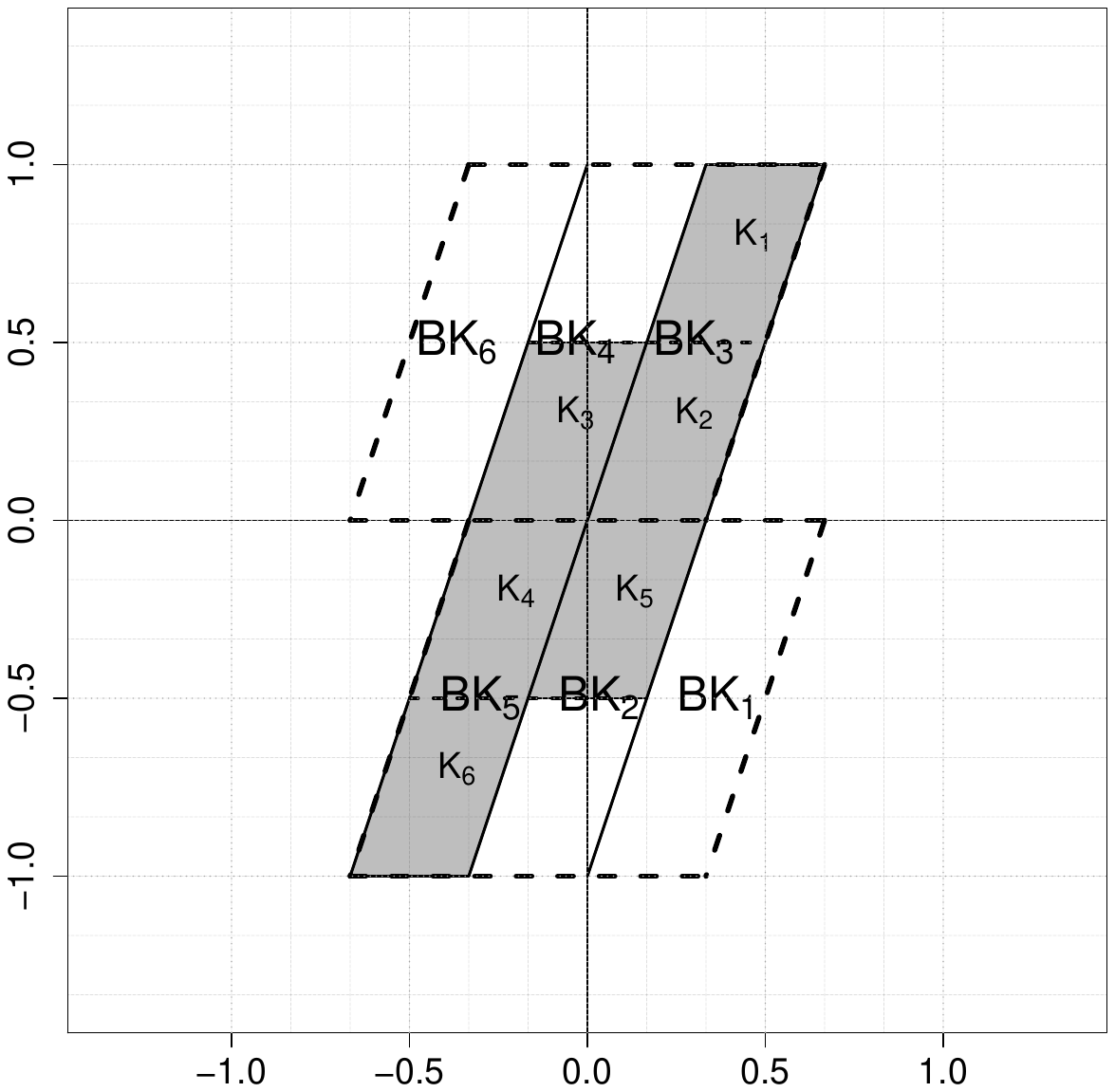}
\caption[Figure 1]{$K_i$ and its $B$-dilation $BK_i$, \ i=1,2,3,4}
\label{fig1.fig}
\end{figure}

\begin{eqnarray*}
&& BK_1=(K_1+\text{\scriptsize$\left(\begin{matrix}0\\-1\end{matrix}\right)$})\cup (K_2+\text{\scriptsize$\left(\begin{matrix}0\\-1\end{matrix}\right)$}), \ BK_2=(K_3+\text{\scriptsize$\left(\begin{matrix}0\\-1\end{matrix}\right)$})\cup K_5, \ BK_3=K_1\cup K_2,\\
&& BK_4=K_3\cup (K_5+\text{\scriptsize$\left(\begin{matrix}0\\1\end{matrix}\right)$}), \ \ BK_5=K_4\cup K_6, \ \ BK_6=(K_4+\text{\scriptsize$\left(\begin{matrix}0\\1\end{matrix}\right)$})\cup (K_6+\text{\scriptsize$\left(\begin{matrix}0\\1\end{matrix}\right)$}),
\end{eqnarray*}
which implies that $\{K_j\}_{j=1}^6$ is an integral self-affine collection. Therefore, $K$ is an integral self-affine $\mathbb{Z}^n$-tiling set with $6$ prototiles. However, this representation is not in its simplest form. We will use the method introduced in Section 2 to represent the set $K$ in its simplest form. At the first step, we get a partition $C_1^{\prime}=\{K_{1i}\}_{i=1}^2$ of $K$ by computing (\ref{e-2-e}) for $m=1$ (see Figure \ref{fig2.fig}).

\begin{figure}[htpb]

 \begin{center}
 \begin{tabular}{cc}
    \hspace{-1in} \includegraphics[trim=0mm 15mm 0mm 14mm, clip, scale=0.4,angle=0]{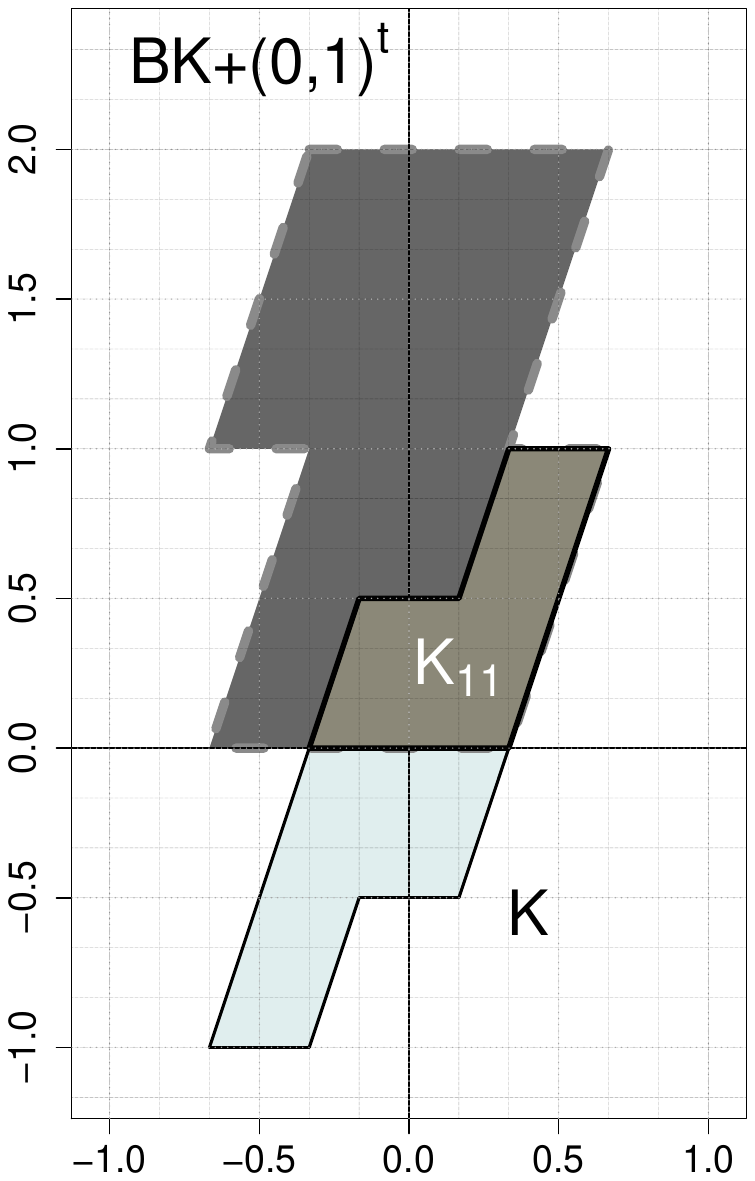} &
    \hspace{-1in} \includegraphics[trim=0mm 15mm 0mm 14mm, clip, scale=0.4,angle=0]{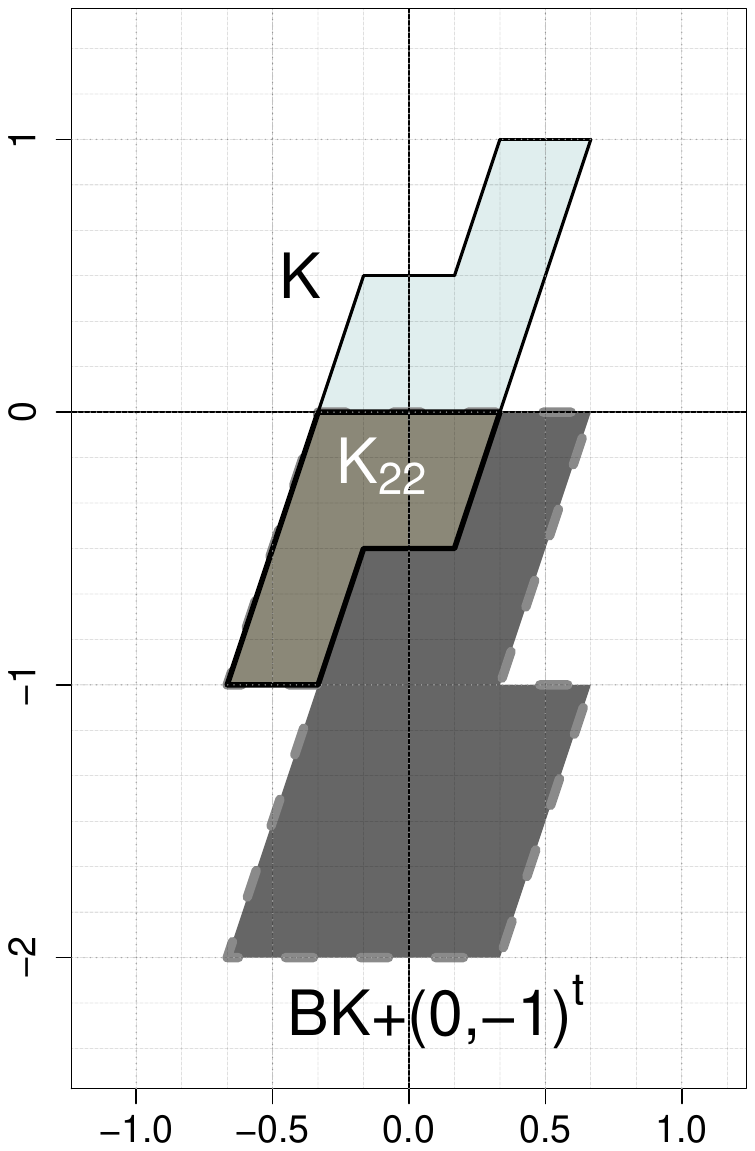}\\
  (a).$K\cap (BK+\text{\scriptsize$\left(\begin{matrix}0\\1\end{matrix}\right)$})$   &  (b).$K\cap (BK+\text{\scriptsize$\left(\begin{matrix}0\\-1\end{matrix}\right)$})$\\
  \end{tabular}
 \caption{The intersection of $K$ and integer translations of $BK$.}
\label{fig2.fig}
 \end{center}
\end{figure}

\begin{eqnarray*}
K_{11}=(BK+\text{\scriptsize$\left(\begin{matrix}0\\1\end{matrix}\right)$})\cap K=K_1\cup K_2\cup K_3, \
K_{12}=(BK+\text{\scriptsize$\left(\begin{matrix}0\\-1\end{matrix}\right)$})\cap K=K_4\cup K_5\cup K_6.
\end{eqnarray*}
It is easy to check that $\{K_{1i}\}_{i=1}^2$ is not an integral self-affine collection. Thus, we need to decompose $K_{1i}, \ i=1,2$ further using (\ref{e-2-e}) (see Figure \ref{fig3.fig}) and we have
\begin{eqnarray*}
&& K_{21}=(B^2K+\text{\scriptsize$\left(\begin{matrix}1\\1\end{matrix}\right)$})\cap K_{11}=K_1\cup K_2, \
K_{22}=(B^2K+\text{\scriptsize$\left(\begin{matrix}-1\\-1\end{matrix}\right)$})\cap K_{11}=K_3, \\
&& K_{23}=(B^2K+\text{\scriptsize$\left(\begin{matrix}-1\\-1\end{matrix}\right)$})\cap K_{12}=K_4\cup K_6, \
K_{24}=(B^2K+\text{\scriptsize$\left(\begin{matrix}1\\1\end{matrix}\right)$})\cap K_{12}=K_5.
\end{eqnarray*}
Furthermore, $\{K_{2i}\}_{i=1}^4$ is an integral self-affine collection since
\begin{eqnarray*}
&& BK_{21}=(K_{21}+\text{\scriptsize$\left(\begin{matrix}0\\-1\end{matrix}\right)$})\cup (K_{21}+\text{\scriptsize$\left(\begin{matrix}0\\-1\end{matrix}\right)$})\cup K_{24}, \ BK_{22}=K_{21}, \\
&& BK_{23}=(K_{23}+\text{\scriptsize$\left(\begin{matrix}0\\1\end{matrix}\right)$})\cup K_{22}\cup (K_{24}+\text{\scriptsize$\left(\begin{matrix}0\\1\end{matrix}\right)$}), \ \ \ BK_{24}=K_{23}.
\end{eqnarray*}

\begin{figure}[htpb]
 \begin{center}
  \begin{tabular}{cc}
 \hspace{-1in} \includegraphics[trim=0mm 8mm 0mm 22mm, clip, scale=0.4,angle=0]{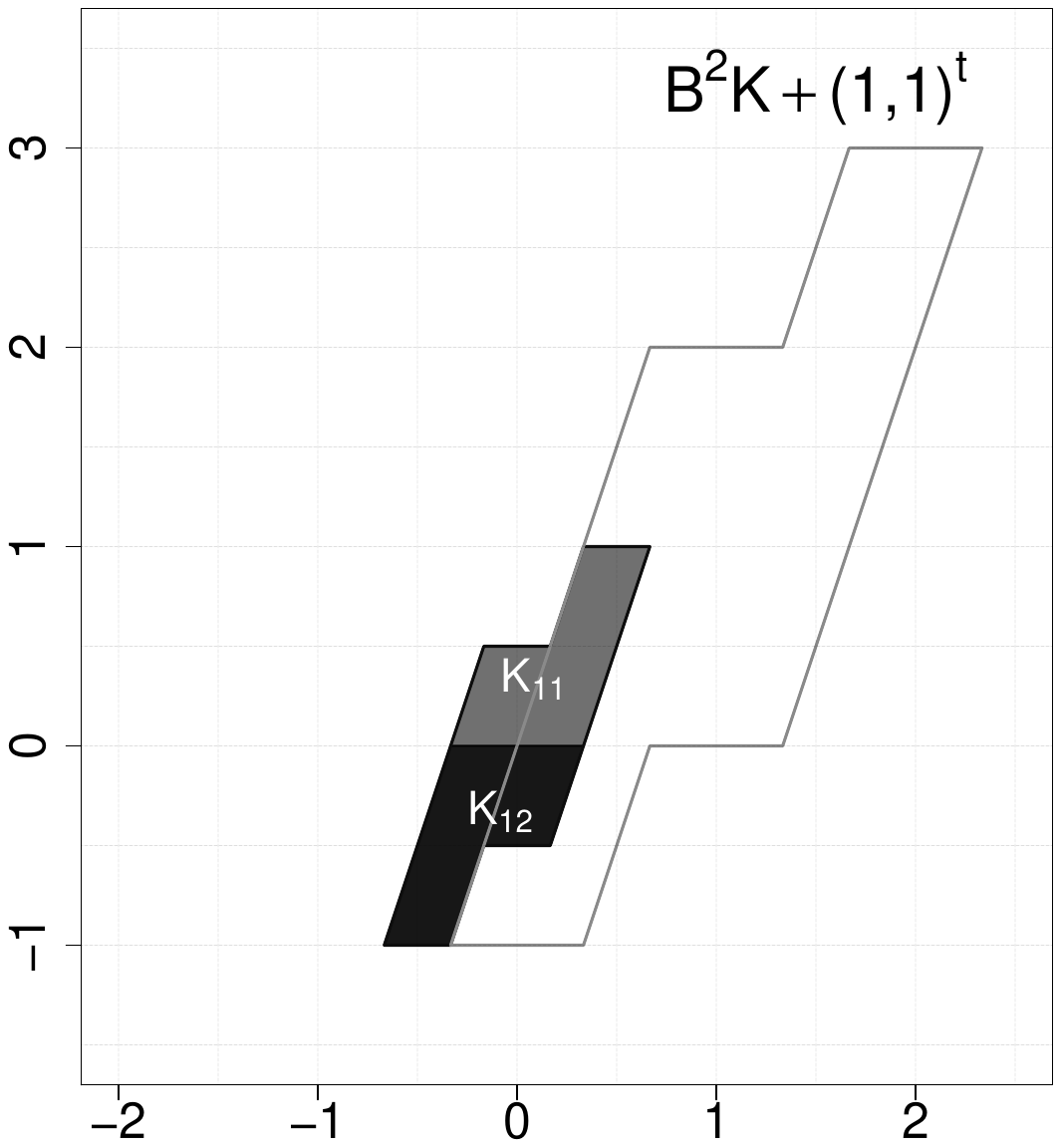} &
 \hspace{-1in} \includegraphics[trim=0mm 8mm 0mm 22mm, clip, scale=0.4,angle=0]{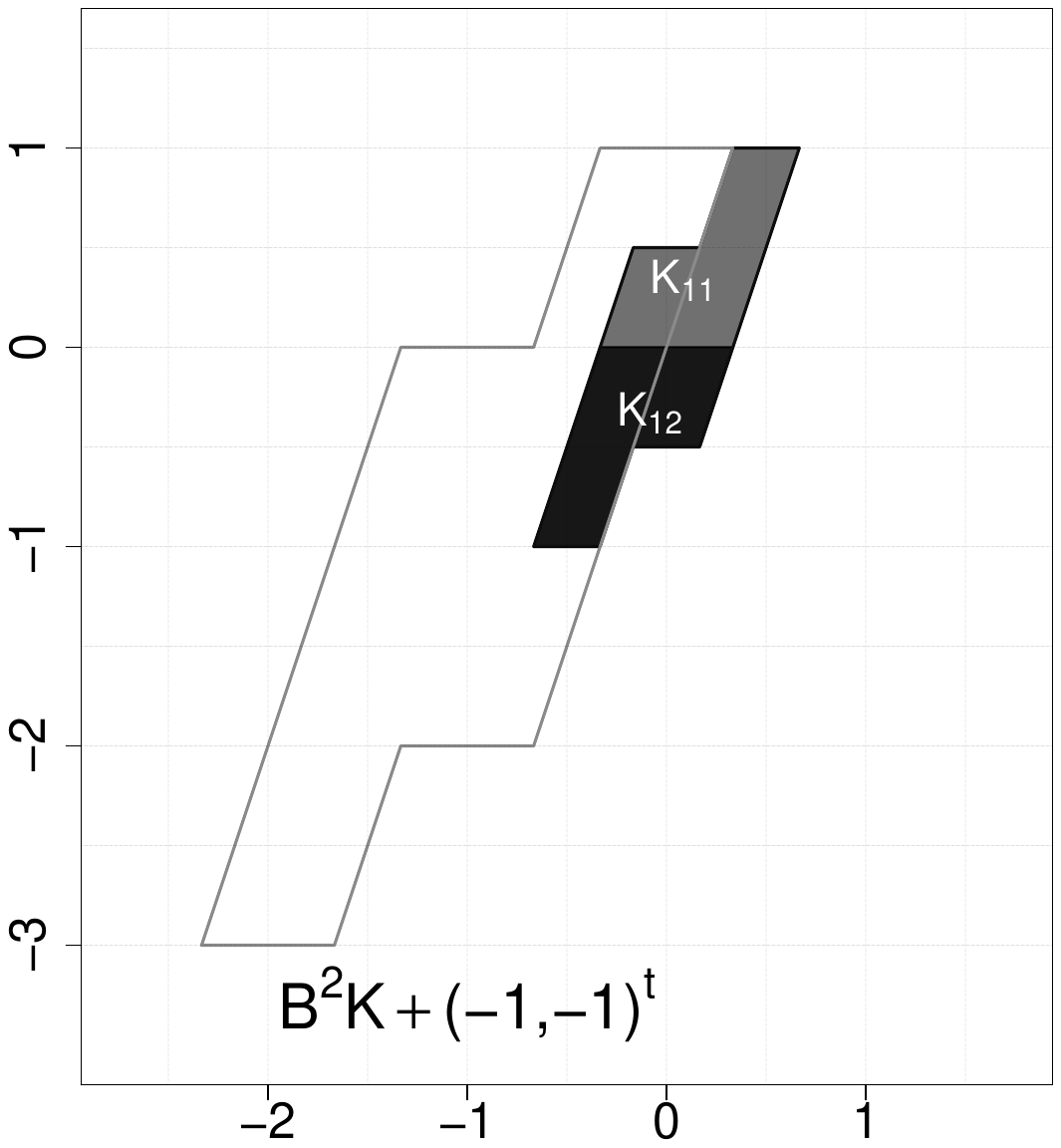}\\
(a) $K_{1i}\cap (B^2K+\text{\scriptsize$\left(\begin{matrix}1\\1\end{matrix}\right)$}), \ i=1,2$   &
(b) $K_{1i}\cap (B^2K+\text{\scriptsize$\left(\begin{matrix}-1\\-1\end{matrix}\right)$}), \ i=1,2$ \\
  \end{tabular}
 \end{center}
\caption{The intersection of $K_{1i},\ i=1,2$ and integer translations of $B^2K$}\label{fig3.fig}
\end{figure}

Therefore, $K=\bigcup\limits_{i=1}^4 K_{2i}$ is an integral self-affine $\mathbb{Z}^2$-tiling set with $4$ prototiles and this representation is in its simplest form. $\Box$

\bigskip

It has been shown in \cite{GHA} that the theory of interal self-affine multi-tiles is closely related to the theory of wavelets. We also considered in \cite{FG} the problem of constructing wavelet sets using integral self-affine multi-tiles and gave a sufficient condition for an integral self-affine $\mathbb{Z}^n$-tiling set with multi-prototiles to be a scaling set. The example below shows that some wavelet sets cannot be constructed using integral self-affine $\mathbb{Z}^n$-tiling sets with multi-prototiles as was done in \cite{FG}.

\bigskip

%Suppose $A$ is an $n\times n$ real expansive matrix, i.e. one with real entries and all its eigenvalues $\lvert\lambda_i(A)\rvert>1$.
%An A-dilation wavelet is a measurable function $\psi\in L^2(\mathbb{R}^n)$ such that the set
%\begin{equation*}
%\{\lvert \det A\rvert^{\frac{j}{2}}\psi(A^jx-k):\ j\in\mathbb{Z},\ k\in\mathbb{Z}^n\}
%\end{equation*}
%forms an othonormal basis for $L^2(\mathbb{R}^n)$. For any function
%$f\in L^1(\mathbb{R}^n)\cap L^2(\mathbb{R}^n)$, its Fourier transform is defined by
%\begin{equation*}
%\mathcal{F}(f)(\xi)=\hat{f}(\xi)=\int_{\mathbb{R}^n}e^{-2\pi ix\cdot\xi}f(x)dx,
%\end{equation*}
% where $x\cdot\xi$ is the standard inner product of the vectors $x,\xi\in\mathbb{R}^n$. The inverse Fourier transform will be denoted as $\mathcal{F}^{-1}$. An $A$-dilation wavelet %set is a measurable set $Q\subseteq\mathbb{R}^n$ such that $\mathcal{F}^{-1}(\chi_Q)$ is an $A$-dilation wavelet. If an $A$-dilation wavelet $\mathcal{F}^{-1}(\chi_Q)$ is %associated with a multiresolution analysis (MRA), then the set $Q\subseteq\mathbb{R}^n$ is called an $A$-dilation MRA wavelet set. A measurable set
%$K$ is called an $A$-dilation scaling set (resp. MRA scaling set) if $Q=BK\setminus K$ is an $A$-dilation wavelet set (resp. MRA wavelet set), where $B=A^t$.

\begin{ex}
In dimension one, consider the set $K=[-a,1-a]$ where $0<a<1$ associated with $B=2$. Then $K$ is an integral self-affine multi-tile if and only if $a\in\mathbb{Q}$.
\end{ex}

\medskip

\begin{proof}
It has been proved in \cite{FGA} that $K$ is a $2$-dilation MRA scaling set and the set $Q:=2K\setminus K$ is a $2$-dilation MRA wavelet set.
Obviously, $K$ is a $\mathbb{Z}$-tiling set. We will divide into two cases to prove our claim.

\noindent Case 1: $a\in\mathbb{Q}$. Then
$a=\frac{p}{q},\ \rm{for \ some} \ p,\ q\in\mathbb{N}$ and $(p,q)=1,\ p<q$ since $0<a<1$. In this case,
$K=[-\frac{p}{q},1-\frac{p}{q}]=[-\frac{p}{q},\frac{q-p}{q}]$. Let
$$
K_i=[-\frac{p-i+1}{q},-\frac{p-i}{q}], \ i=1,\dotsc, q.$$
Then $K=\bigcup\limits_{i=1}^{q}K_i$, where the union is measure disjoint and we have
$$\begin{array}{crl}
BK_i&=&2[-\frac{p-i+1}{q},-\frac{p-i}{q}]
=[-\frac{2p-2i+2}{q},-\frac{2p-2i}{q}]\\
&=&[-\frac{2(p-i+1)}{q},-\frac{2p-2i+1}{q}]\cup[-\frac{2p-2i+1}{q},-\frac{2(p-i)}{q}].\end{array}$$
Note that $$-\frac{2(p-i+1)}{q}\in\Big\{-\frac{p}{q},-\frac{p-1}{q},...,-\frac{1}{q},0,\frac{1}{q},...,\frac{q-p-1}{q}\Big\}+\mathbb{Z},$$ and
$$-\frac{2(p-i)}{q}\in\Big\{-\frac{p-1}{q},-\frac{p-2}{q},...,0,\frac{1}{q},...,\frac{q-p}{q}\Big\}+\mathbb{Z}.$$
Thus, we have $BK_i=(K_{j_1}+\ell_1)\cup(K_{j_2}+\ell_2)$ for some
$\ell_1,\ \ell_2\in\mathbb{Z}$, \ \rm{and} \ $j_1,\ j_2\in\{1,2,...,q\}$.
This proves that the collection $\{K_i\}_{i=1}^q$ is an integral self-affine collection.
Therefore, $K = \bigcup\limits_{i=1}^{q}K_i$ is an integral self-affine ${\mathbb Z}$-tiling set with $q$ prototiles.

\noindent Case 2: $a\notin\mathbb{Q}$.
 Let $\lceil a \rceil$ denote the minimum integer not less than $a$. Note that for any $m\ge 0$, $(2^mK + \lceil 2^ma - a \rceil ) \cap K \ncong \emptyset$ since $-a \le \lceil 2^ma - a \rceil  - 2^ma < 1-a$, which
implies that the collection of each step in Section 2 must contain the sets with left endpoints $ \lceil 2^ma - a \rceil  - 2^ma $. Then it is infinite since $a\notin\mathbb{Q}$. This shows that the iteration
 will go on infinitely and thus $K=[-a,1-a]$ is not an integral self-affine multi-tile.
\end{proof}

%\section*{Acknowledgements}
%The authors would like to thank the reviewers for their valuable comments and suggestions which were very helpful for improving the paper.

\end{document}